\documentclass{amsart}
\usepackage{graphicx,euscript,amsmath,bm,enumerate,lipsum}
\usepackage[alphabetic]{amsrefs}
\newtheorem{theorem}{Theorem}[section]

\theoremstyle{definition}
\newtheorem{remark}{Remark}

\newtheorem*{lemma*}{Lemma}

\newcommand{\myfig}[3][]{
 \begin{figure}
 \begin{center}
 {\mbox{\includegraphics[#1]{#2.eps}}}
 \end{center}
 \caption{\label{#2}#3}
 \end{figure}}

\def\vc#1{\mathbf #1}
\def\mymatrix#1{\begin{bmatrix}#1\end{bmatrix}}

\def\O{{\mathcal O}}
\def\F{{\mathcal F}}
\long\def\ignore#1{}
\def\tt{\theta}

\def\cc{\gamma}

\def\H{{\mathcal H}}

\def\nd{\noindent}
\def\fref#1{Figure~\ref{#1}}
\def\aa{\alpha}
\def\ss{\sigma}

 \providecommand{\Pic}{\mathop{\rm Pic}\nolimits}

\def\Bbb#1{{\mathbb #1}}

\def\E{{\mathcal E}}

\def\bb{\beta}
\def\dd{\delta}

\def\CC{{\Gamma}}
\def\GG{{\Gamma}}

\def\LL{\Lambda}

\def\K{{\mathcal K}}

\def\L{{\mathcal L}}

\def\proj{{\rm{proj}}}

\def\H{{\mathcal H}}

\renewcommand{\odot}{\mathbin{\mathchoice
  {\xcirc\scriptstyle}
  {\xcirc\scriptstyle}
  {\xcirc\scriptscriptstyle}
  {\xcirc\scriptscriptstyle}
}}
\newcommand{\xcirc}[1]{\vcenter{\hbox{$#1\circ$}}}

\def\P{{\mathcal P}}

\begin{document}
\title{Enriques surfaces and an Apollonian packing in eight dimensions}
\author{Arthur Baragar}
\begin{abstract}
We call a packing of hyperspheres in $n$ dimensions an Apollonian sphere packing if the spheres intersect tangentially or not at all; they fill the $n$-dimensional space; and every sphere in the packing is a member of a cluster of $n+2$ mutually tangent spheres (and a few more properties described herein).  In this paper, we describe an Apollonian packing in eight dimensions that naturally arises from the study of generic nodal Enriques surfaces.  The $E_7$, $E_8$ and Reye lattices play roles.  We use the packing to generate an Apollonian packing in nine dimensions, and a cross section in seven dimensions that is weakly Apollonian.   Maxwell described all three packings but seemed unaware that they are Apollonian.  The packings in seven and eight dimensions are different than those found in an earlier paper.  In passing, we give a sufficient condition for a Coxeter graph to generate mutually tangent spheres, and use this to identify an Apollonian sphere packing in three dimensions that is not the Soddy sphere packing.
\end{abstract}
 %, 11H31, 06B99}
\address{Department of Mathematical Sciences, University of Nevada, Las Vegas, NV 89154-4020}
\email{baragar@unlv.nevada.edu}
\thanks{\nd \LaTeX ed \today.}

\maketitle

\section*{Introduction}  That there is a connection between the Apollonian packing and rational curves on algebraic surfaces has only recently been explored (see \cite{Dol16a, Bar17}).  The connection so far has seemed a bit distant.  In this paper, we show a rather spectacular connection.  Let $X$ be an Enriques surface that contains a smooth rational curve (a {\it nodal} curve), but is otherwise generic.  Let $\LL$ be its Picard group, modulo torsion.  Then $\LL$ is isomorphic to the Enriques lattice $E_{10}$, which is an even unimodular lattice of signature $(1,9)$.  Since $\LL\otimes \Bbb R$ is a Lorentz space $\Bbb R^{1,9}$, it has a 9-dimensional copy of hyperbolic space $\Bbb H^9$ naturally imbedded in it.  Any nodal curve on $X$ has self intersection $-2$, so represents a plane in $\Bbb H^9$, which in the Poincar\'e upper-half hyperspace model is represented by an 8-dimensional hemi-sphere.  The boundary $\partial \Bbb H^9$ of $\Bbb H^9$, not including the point at infinity, is isomorphic to $\Bbb R^8$, and its intersection with a hyperbolic plane in $\Bbb H^9$ is a 7-dimensional hypersphere.  In this way, the set of nodal curves on $X$ gives us a configuration of hyperspheres in $\Bbb R^8$.  In this paper, we show that this configuration is an Apollonian packing, by which we mean the hyperspheres intersect tangentially or not at all, they fill $\Bbb R^8$, the packing is crystallographic, and it satisfies the fundamental property that makes it Apollonian, which is that every sphere in the configuration is a member of a  cluster of ten hyperspheres that are mutually tangent.   (See Section~\ref{s1.3} for precise definitions.)

Though the connection to Enriques surfaces is fascinating, this paper is really a study of sphere packings.
The underlying algebraic geometry is explained in \cite{All18, Dol16b}.  Our starting point is
the following Coxeter graph (also known as a  Dynkin diagram):
\begin{equation}\label{C2}%\begin{center} \label{C2}
\includegraphics{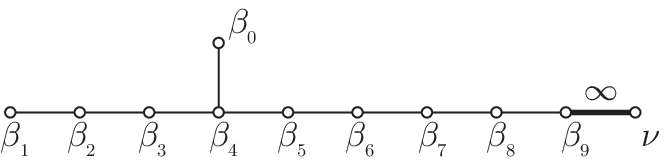}
%\end{center}
\end{equation}
Some may recognize that it generates the Reye Lattice.  The graph also appears in Maxwell's paper \cite[Table II, $N=10$]{Max82}, in which it is described as generating a sphere packing, meaning the spheres intersect tangentially or not at all.  It is one of many examples of what are sometimes called Boyd-Maxwell packings (see also \cite{CL15}).
Up until recently (see \cite{Bar18, Bar19}), it was believed that no Boyd-Maxwell packing in dimension $n\geq 4$ is Apollonian (see the {\it Mathematical Review} for \cite{Boy74}, \cite[p.~356]{LMW02}, and \cite[p.~33]{Dol16a}).  In \cite{Bar19}, we describe a different Apollonian packing (though not a Boyd-Maxwell packing) in eight dimensions.  The packing in this paper is more efficient, in the sense that its residual set is a subset of the residual set of the packing in \cite{Bar19}.  This also shows that there can be different Apollonian packings in the same dimension.

In Section \ref{s1.3}, we give a precise definition of a generalized Apollonian packing.  This is the author's attempt to define what was likely meant or understood by Boyd, Maxwell, and their contemporaries.

A cross section of a sphere packing gives a sphere packing in a lower dimension.  By taking a co-dimension one cross section perpendicular to nine of the ten mutually tangent spheres of our new packing in eight dimensions, we get a {\em weakly} Apollonian packing in seven dimensions, meaning there is a cluster of nine mutually tangent spheres, but not all spheres are a member of such a packing.  It too is described by Maxwell \cite[Table II, $N=9$, last line, first graph]{Max82}.  The resulting packing is different from and more efficient than the example in \cite{Bar19}, which is Apollonian.  

For Euclidean lattices, eight dimensions is special as it admits the even unimodular lattice $E_8$.  The uniradial sphere packing (i.e. all spheres have the same radius) with spheres of radius $1/\sqrt{2}$ and centered at vertices of the lattice gives the densest uniradial sphere packing in eight dimensions \cite{Via17}.  With an appropriate choice of point at infinity, the Apollonian packing of this paper includes this uniradial sphere packing.

In \cite{Bar18}, we generated Apollonian packings in dimension $n$ from uniradial sphere packings in dimension $n-1$.  Using a similar procedure and the $E_8$ lattice, we generate an Apollonian packing in nine-dimensions.  This too appears in Maxwell's paper \cite[Table II, $N=11$, second graph]{Max82}.   

This paper was directly inspired by a paper by Daniel Allcock \cite{All18}, who reproved a result (see Theorem \ref{tAllcock} below) that appears in \cite{CD89}.   Dolgachev attributes the original proof to Looijenga, based on an incomplete proof by Coble \cite{Cob19} from 1919.

The descriptions of the Apollonian packings of this paper do not require any background in algebraic geometry.  We will therefore keep the algebraic geometry to a minimum, restricting it as much as possible to Subsection \ref{sAG} and to remarks.

With respect to the organization of this paper,  Section \ref{s1} on Background includes a definition of higher dimensional analogs of the Apollonian circle packings, and the main result concerning Enriques surfaces from which the packing is derived; Section \ref{s2} is devoted to producing the Apollonian packing in eight dimensions;  In Section \ref{s4} we look at a cross section; and in Section \ref{s5} we build an example in nine-dimensions.  In Section \ref{s3} we describe a sufficient condition for a packing to have the (weak) Apollonian property and use it to find examples in Maxwell's list of packings, including a sphere packing in three dimensions that has the Apollonian property but is not the Soddy sphere packing.

\subsection*{Acknowledgements} The author is grateful to {\it zbMath} for the fortuitous request that he review Allcock's paper, and to Daniel Allcock who was kind and patient enough to explain details in that paper.  The author is also grateful to David Boyd for the many conversations through the years, and to Alex Kontorovich and Daniel Lautzenheiser for some useful conversations.  This material is based upon work supported by the National Science Foundation under Grant No. DMS-1439786 and the Alfred P. Sloan Foundation award G-2019-11406 while the author was in residence at the Institute for Computational and Experimental Research in Mathematics in Providence, RI, during the Illustrating Mathematics program.  Finally, the author wishes to thank his home institution, UNLV, for its sabbatical assistance during the Fall of 2019.  

\section{Background}\label{s1}

\subsection{The pseudosphere in Lorentz space}   Hyperspheres in $\Bbb R^n$ can be represented by $(n+2)$-dimensional vectors.  Boyd calls such coordinates  {\it polyspherical} coordinates, and attributes them to Clifford and Darboux from the late 19th century \cite{Boy74}.  The more modern interpretation is that they represent planes in $\Bbb H^{n+1}$ imbedded in an $(n+2)$-dimensional Lorentz space, which in turn represent hyperspheres on the boundary $\partial \Bbb H^{n+1}$.  We refer the reader to \cite{Rat06} for more details.

Let us set $N=n+2$.

Given a symmetric matrix $J$ with signature $(1,N-1)$, we define the Lorentz space, $\Bbb R^{1,N-1} $ to be the set of $N$-tuples over $\Bbb R$ equipped with the negative Lorentz product
\[
\vc u\cdot \vc v=\vc u^TJ\vc v.
\]
The surface $\vc x\cdot \vc x=1$ is a hyperboloid of two sheets.  Let us distinguish a vector $D$ with $D\cdot D>0$ and select the sheet $\H$ by:
\[
\H: \qquad \vc x\cdot \vc x=1, \qquad \vc x\cdot D>0.
\]
We define a distance on $\H$ by
\[
\cosh(|AB|)=A\cdot B.
\]
Then $\H$ equipped with this metric is a model of $\Bbb H^{N-1}$, sometimes known as the {\it vector model}.   Equivalently, one can define
\[
V=\{\vc x\in \Bbb R^{1,N-1}: \vc x\cdot \vc x>0\}
\]
and $\H=V/\Bbb R^*$, together with the metric defined by
\[
\cosh(|AB|)=\frac{A\cdot B}{|A||B|},
\]
where $|\vc x|=\sqrt{\vc x\cdot \vc x}$ for $\vc x\in V$.  For $\vc x\cdot \vc x<0$, we define $|\vc x|=i\sqrt{-\vc x\cdot \vc x}$.

Hyperplanes in $\H$ are the intersection of $\H$ with hyperplanes $\vc n\cdot \vc x=0$ in $\Bbb R^{1,N-1}$.  Such a plane intersects $\H$ if and only if $\vc n\cdot \vc n<0$.  Let $H_{\vc n}$ represent both the plane $\vc n\cdot \vc x=0$ in $\Bbb R^{1,N-1}$ and its intersection with $\H$.  The direction of $\vc n$ distinguishes a half space
\[
H_{\vc n}^+=\{\vc x: \vc n\cdot \vc x>0\},
\]
in either $\Bbb R^{1,N-1}$ or $\H$.

The angle $\tt$ between two intersecting planes $H_{\vc n}$ and $H_{\vc m}$ in $\H$ is given by
\begin{equation}\label{eq1}
|\vc n||\vc m|\cos \tt=\vc n\cdot \vc m,
\end{equation}
where $\tt$ is the angle in $H_{\vc n}^+\cap H_{\vc m}^+$.  If $|\vc n\cdot \vc m|=||\vc n||\vc m||$, then the planes are tangent at infinity.  If $|\vc n\cdot \vc m|>||\vc n||\vc m||$, then the planes do not intersect, and the quantity $\psi$ in  $|n||m|\cosh \psi=|\vc n\cdot \vc m |$ is the shortest hyperbolic distance between the two planes.

The group of isometries of $\H$ is given by
\[
\O^+(\Bbb R)=\{T\in M_{N\times N}: \hbox{$T\vc u\cdot T\vc v=\vc u\cdot \vc v$ for all $\vc u, \vc v\in \Bbb R^{1,N-1}$, and $T\H=\H$.}\}.
\]
Reflection in the plane $H_{\vc n}$ is given by
\[
R_{\vc n}(\vc x)=\vc x-2\proj_{\vc n}(\vc x)=\vc x-2\frac{\vc n\cdot \vc x}{\vc n\cdot \vc n}\vc n.
\]
The group of isometries is generated by the reflections.

Let $\partial \H$ represent the boundary of $\H$, which is a $(N-2)$-sphere.  It is represented by $\L^+/\Bbb R^+$ where
\[
\L^+=\{\vc x\in \Bbb R^{1,N-1}: \vc x\cdot \vc x=0, \vc x\cdot D>0\}.
\]
Given an $E\in \L^+$, let $\partial \H_E=\partial \H \setminus E\Bbb R^+$.  Then $\partial \H_E$ equipped with the metric $|\cdot |_E$ defined by
\[
|AB|_E^2=\frac{2A\cdot B}{(A\cdot E)(B\cdot E)}
\]
is the $(N-2)$-dimensional Euclidean space that is the boundary of the Poincar\'e upper half hyperspace model of $\H$ with $E$ the point at infinity.  In $\partial \H_E$, the plane $H_{\vc n}$ is represented by an $(N-3)$-sphere, which we denote with $H_{\vc n,E}$ (or just $H_{\vc n}$ if $E$ is understood, or sometimes just $\vc n$).

The curvature (the inverse of the radius, together with a sign) of $H_{\vc n, E}$ is given by the formula
\[
\frac{\vc n\cdot E}{||\vc n||}
\]
using the metric $|\cdot |_E$ \cite{Bar18}.  Here, $||\vc n||=-i|\vc n|=\sqrt{-\vc n\cdot \vc n}$.    By choosing a suitable orientation for $\vc n$, we get the appropriate sign for the curvature.

\subsection{Coxeter graphs}

Coxeter graphs can represent lattices in $\Bbb R^n$ or $\Bbb R^{1,n-1}$, or groups of isometries of $\Bbb S^{n-1}$ or $\Bbb H^{n-1}$.  Each node in a Coxeter graph represents a plane, which can be represented by its normal vector.  Two nodes are not connected if their corresponding planes are perpendicular.  A regular edge between two nodes indicates that those planes intersect at an angle of $\pi/3$, so their normal vectors are at an angle of $2\pi/3$ (some ambiguity here).  If the angle between two planes is $\pi/4$, then we indicate that with an edge subscripted (or superscripted) with a 4, or sometimes a double edge.  In general, an edge with a superscript of $m$ (or of multiplicity $m-2$) means the order of the composition of reflections in the two planes represented by the two nodes is $m$.  A thick line, like the one in (\ref{C2}), means the two planes are parallel, and this is sometimes indicated with a superscript of $\infty$.  A dotted edge means the two planes are ultraparallel.

The Coxeter graph represents the lattice $\vc v_1\Bbb Z\oplus \cdots \oplus \vc v_n\Bbb Z$, where the $\vc v_i$ are the nodes of the graph.  Note that not all graphs give lattices, as the vectors may be linearly dependent.  When the nodes are linearly independent (and sometimes when they are not), we get the bilinear form defined by the {\em incidence matrix} $[\vc v_i\cdot \vc v_j]$.  We can normalize the vectors $\vc v_i$ (say) so that $\vc v_i\cdot \vc v_i=-1$.  Then the angles noted by the edges define $\vc v_i\cdot \vc v_j$, where we take the angle between the normal vectors of the planes to be obtuse (so $\vc v_i\cdot\vc v_j\geq 0$).

The Coxeter graph can also represent a group, the {\it Weyl} group $\langle R_{\vc v_1},...,R_{\vc v_n}\rangle$.  When representing a group, it is often necessary for the vectors to be linearly dependent.

We will explain the notion of {\em weights} in the next subsection, which will give us an example to investigate.

\subsection{The Apollonian circle packing}  The Apollonian circle packing is a well known object and we assume the reader is already familiar with it.  The goal of this section is to think of it in a way that more naturally generalizes to higher dimensions.

Let us consider the strip version of the Apollonian circle packing shown in \fref{figApc}, and think of the picture as lying on the boundary $\partial \Bbb H^3$ of hyperbolic space.  Then each circle (and the two lines) represent a plane in $\Bbb H^3$, which in turn are represented by their normal vectors.  Let us represent four of the hyperbolic planes with the vectors $\vc e_i$ for $i=1,...,4$, as shown in \fref{figApc}.  We orient $\vc e_i$ so that $H_{\vc e_i}^+$ contains $H_{\vc e_j}$ for $i\neq j$, and assign them norms of $-2$, meaning $\vc e_i\cdot \vc e_i=-2$.  Then, by the tangency conditions, $\vc e_i\cdot \vc e_j=2$ for $i\neq j$ (see equation (\ref{eq1})).  We define $J_4=[\vc e_i\cdot \vc e_j]$, which has signature $(1,3)$.  (It is easy to see that $4$ and $-4$ are eigenvalues, the latter with a 3-dimensional eigenspace.)  Thus, $J_4$ defines a Lorentz product in $\Bbb R^{1,3}$.

\myfig{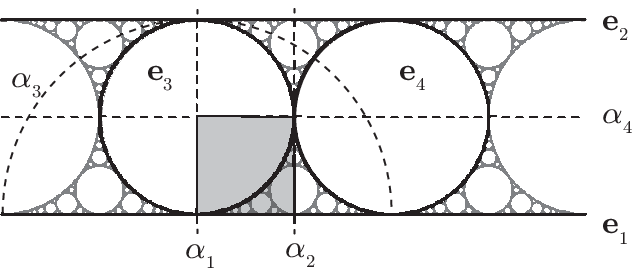}{The strip version of the Apollonian packing.  The dotted curves $\aa_i$ represent symmetries of the packing.  The fundamental domain $\F_4$ for the group $\O^+_{\LL_4}$ is the region above the shaded square and above the plane $H_{\aa_3}$. }

Note that $\vc e_i\cdot(\vc e_i+\vc e_j)=0$, so $\vc e_i+\vc e_j$ is the point of tangency between the circles represented by $\vc e_i$ and $\vc e_j$.  Thus, our point at infinity is $E=\vc e_1+\vc e_2$, and the circle represented by $\vc e_i$ is $H_{\vc e_i,E}$.

There are obvious symmetries of the Apollonian packing, as shown in \fref{figApc}.  In the plane, the symmetries are reflection in the lines $\aa_1$, $\aa_2$, and $\aa_4$, and inversion in the circle $\aa_3$.  The Apollonian packing is the orbit of $\vc e_1$ under the action of the group generated by these symmetries.

Thought of as actions in $\Bbb H^3$, the symmetries are reflection in the planes $H_{\aa_i}$, so the packing is the $\CC_4$-orbit of $H_{\vc e_1}$, where
\[
\CC_4=\langle R_{\aa_1}, R_{\aa_2}, R_{\aa_3}, R_{\aa_4}\rangle.
\]
It is fairly easy to solve for $\aa_i$ in the basis $\vc e=\{\vc e_1,...,\vc e_4\}$.  For example, we note that $\aa_1\cdot \vc e_i=0$ for $i\neq 4$, since it is perpendicular to those planes.  Solving, we get $\aa_1=(1,1,1,-1)$, up to scalars.  The others are: $\aa_2=(0,0,-1,1)$, $\aa_3=(0,-1,1,0)$, and $\aa_4=(-1,1,0,0)$.
%For each, we chose a minimal solution in $\LL_4$; we explain our choice of orientation in a moment.

We let $\LL_4=\vc e_1\Bbb Z\oplus ...\oplus \vc e_4\Bbb Z$, which we call the {\it Apollonian lattice}.  Let
\[
\O^+_{\LL_4}=\{T\in \O^+: T\LL_4=\LL_4\}.
\]
It is straight forward to verify that $R_{\vc e_i}$ and $R_{\vc \aa_i}\in \O^+_{\LL_4}$.  Let $G_4=\langle \CC_4, R_{\vc e_1}\rangle$.  A fundamental domain $\F_4$ for this group is the region bounded by the four planes $H_{\aa_i}$, the pane $H_{\vc e_1}$, and with cusp $E$ at the point at infinity:
%, as shown in \fref{figAp}:
\[
\F_4=H^+_{\aa_1}\cap \cdots \cap H^+_{\aa_4} \cap H^+_{\vc e_1}.
\]
Since $\F_4$ has finite volume, we know $G_4$ has finite index in $\O^+_{\LL_4}$, and it is not hard to verify that the two are equal.

Let us also use this example to learn a little about Coxeter graphs.  The Coxeter graph for the planes that bound $\F_4$ is the following:
\begin{equation}\label{C3}
\includegraphics{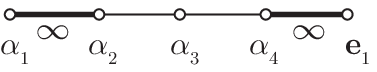}
\end{equation}

We note, from our earlier calculations, that $\aa_i\cdot \aa_i=-8$.   From the Coxeter graph, we therefore get $\aa_1\cdot \aa_2=8$, $\aa_2\cdot \aa_3=4$, $\aa_3\cdot \aa_4=4$, and all other products $\aa_i\cdot \aa_j$ for $i<j$ are zero.  This gives us the matrix
\[
J_\aa=[\aa_i\cdot \aa_j]=\mymatrix{-8&8&0&0 \\ 8&-8&4&0 \\ 0&4&-8&4 \\ 0&0&4&-8}.
\]
Our Lorentz product, in this basis, is $\vc x \cdot \vc y=\vc x^tJ_{\aa}\vc y$.  Let $\LL_{\aa}=\aa_1\Bbb Z\oplus \cdots \oplus \aa_4\Bbb Z$, so $\LL_{\aa}\subset \LL_{4}$.  Since $\det(J_\aa)=4\det(J_4)$, the index is two.

The fundamental domain $\F_4$ has five vertices:  The point $E$ at infinity, and the four vertices on $H_{\aa_3}$, the hemisphere above the dotted circle (see \fref{figApc}).  Let $w_i$ be the intersection of the three $\aa_j$ with $j\neq i$, so $w_i$ is a vertex of $\F_4$ for $i=1$ and $2$, and $w_3$ is the point at infinity.  The $w_i$'s satisfy $w_i\cdot \aa_j=0$ for $i\neq j$, so in the basis $\aa=\{\aa_1,...,\aa_4\}$, are the rows of $J_{\aa}^{-1}$, up to scalars.  To express these points in the basis $\vc e$, we multiply on the left by the change of basis matrix $Q=[\aa_1|\aa_2|\aa_3|\aa_4]$ whose columns are $\aa_i$; the columns of the resulting matrix are the $w_i$'s.  The $w_i$'s are called {\em weights}.  The incidence matrix for the $w_i$s is $[w_i\cdot w_j]=J_\aa^{-1}J_\aa (J_\aa^{-1})^t=(J_\aa^{-1})^t=J_\aa^{-1}$.  The first and second diagonal elements of $J_\aa^{-1}$ are positive, indicating that $w_1$ and $w_2$ lie in $\H$.  The third diagonal element is $0$, indicating $w_3$ lies on $\partial \H$, so is a cusp (it is the point $E$ at infinity).  The last diagonal element is negative, so $w_4$ represents a plane in $\H$.  That plane is perpendicular to $H_{\aa_1}$, $H_{\aa_2}$, and $H_{\aa_3}$, so is the plane $H_{\vc e_1}$.

The other two vertices of $\F_4$ can be found in a similar way by replacing $\aa_4$ with $\vc e_1$.   This is plain to see from the picture (see \fref{figApc}).  Combinatorially, the intersection of any three of the five faces should give us a potential vertex, but since $H_{\aa_4}$ and $H_{\vc e_1}$ are parallel, any combination that includes those two will give at most their point of tangency.

\begin{remark}  The symmetry of the Coxeter graph about the node $\aa_3$ suggests that there is a symmetry that sends $\vc e_1$ to $\aa_1$, etc.  That symmetry is reflection in the diagonal of slope one of the shaded square in \fref{figApc}.  However, since the norms of $\vc e_1$ and $\aa_1$ are different, it is not a symmetry of the lattice $\LL_4$.  There is a different lattice that we could have considered, namely the one generated by the Coxeter graph (\ref{C3}) but with $\aa_i^2=-2$.  There are some advantages to looking at that graph.  Reflection in the diagonal is a symmetry of that lattice.
\end{remark}

\subsection{Generalized sphere packings}\label{s1.3}

A {\it sphere packing} in $\Bbb R^n$ is a configuration of oriented $(n-1)$-spheres that intersect tangentially or not at all.  By oriented, we mean each sphere includes either the inside (a ball) or the outside.  The trivial sphere packing is a sphere and its complement.  We will not consider trivial sphere packings.

Maxwell calls $\P\subset \Bbb R^{1,N-1}$ a {\it packing} if for all $\vc n, \vc n'\in \P$, there exists a positive constant $k$ such that $\vc n\cdot \vc n=-k$ and $\vc n\cdot \vc n'\geq k$ \cite{Max82}.  Given a point $E$ for the point at infinity, the packing $\P$ defines a sphere packing
\[
\P_E=\bigcap_{\vc n\in \P}H_{\vc n,E}^-\subset \partial \H_E \cong \Bbb R^{n}.
\]
We call $\P_E$ a {\it perspective} of $\P$.

For example, the Apollonian packing in $\Bbb R^{1,3}$ is $\P_4=\CC_4(\vc e_1) $ and the strip packing of \fref{figApc} is $\P_{4,\vc e_1+\vc e_2}$.  An Apollonian packing derived from a different initial cluster of four mutually tangent circles is a different perspective of $\P_4$.

We think of $\P$ as defining a {\it cone}
\[
\K_{\P}=\bigcap_{\vc n\in \P}H_{\vc n}^+
\]
in $\Bbb R^{1,N-1}$, or a {\it polyhedron}
\[
\K_{\P}\cap \H
\]
in $\Bbb H^{N-1}$.  The {\it residual set} of a sphere packing is the complement of the sphere packing in $\partial \H_E\cong \Bbb R^n$, and is the intersection of $\K_{\P}$ with $\partial \H_E$.  A sphere packing is {\it maximal} or {\it dense} if there is no space in the residual set where one can place another sphere of positive radius.  It is {\it complete} if the residual set is of measure zero.  A packing that is maximal (respectively complete) in one perspective is maximal (complete) in any perspective.

We call a packing {\it of lattice type} if $\P\Bbb Z$ forms a lattice in $\Bbb R^{1,N-1}$ \cite{Max82}.  We call a packing {\it of general lattice type} if  there exists a lattice $\LL\subset R^{1,N-1}$ so that a scalar multiple of $\vc n$ is in $\LL$ for all $\vc n\in \P$, and $\P$ spans $\Bbb R^{1,N-1}$.

A subgroup $\CC \leq O_\LL^+$ is called {\em geometrically finite} if it has a convex fundamental domain with a finite number of faces.  A packing is called {\it crystallographic} if there exists a geometrically finite group $\CC\leq O_{\LL}^+$ and a finite set $S\subset \LL$ so that
\[
\P=\{\cc(\vc n)/|\cc(\vc n)|: \cc\in \CC, \vc n\in S\}.
\]
As the normality condition is not necessary, we will write $\P=\CC(S)$.  The residual set for the packing is the limit set of $\CC$.

Crystallographic packings are not always of lattice type, as was noted by Maxwell.

A crystallographic packing that is maximal is known to be complete.  Furthermore, the Hausdorff dimension of the residual set is strictly less than $N-2$ \cite{Sul84}.

The packings that are known as Boyd-Maxwell packings are crystallographic packings with the restriction that $\CC$ be a reflective group (i.e. is generated by a finite number of reflections).  Kontorovich and Nakamura include this requirement in their definition of a crystallographic packing \cite{KN2019}.

We say a packing has the {\em weak Apollonian property} if it contains a cluster of $N=n+2$ mutually tangent spheres.  We say it has the {\em  Apollonian property} if every sphere is a member of a cluster of $N$ mutually tangent spheres.  
We call a packing {\em Apollonian} if it is crystallographic, maximal, and has the Apollonian property.   

%Finally, we call a crystallographic packing an {\it Apollonian packing} if it contains a subset $\{\vc e_1,...,\vc e_{N}\}$ such that $\vc e_i\cdot \vc e_i<0$ and $(\vc e_i\cdot \vc e_j)^2=(\vc e_i\cdot \vc e_i)(\vc e_j\cdot \vc e_j)$ for all $i\neq j$.  That is, the packing includes a cluster of $N=n+2$ mutually tangent spheres.  

\begin{remark}\label{r2}  In \cite{Bar18}, we give a different definition for an Apollonian packing.  We begin with a set $\{\vc e_1,...,\vc e_{N}\}$ with $\vc e_i\cdot \vc e_i=-1$ and $\vc e_i\cdot \vc e_j=1$ for $i\neq j$ (so a cluster of $N$ mutually tangent spheres), and define the lattice
\[
\LL_N=\vc e_1\Bbb Z\oplus\cdots\oplus \vc e_{N}\Bbb Z.
\]
We pick $D\in \LL_N$ with $D\cdot D>0$ and such that $D\cdot \vc n\neq 0$ for any $\vc n\in \LL_N$ with $\vc n\cdot \vc n=-1$.  (Such a $D$ exists.) We define
\[
\E_{-1}=\{\vc n\in \LL_N: \vc n\cdot \vc n=-1, \vc n\cdot D>0\}.
\]
We define the cone
\[
\K_\rho=\bigcap_{\vc n\in \E_{-1}} H_{\vc n}^+
\]
and the set
\[
\E_{-1}^*=\{\vc n\in \E_{-1}: \hbox{$H_{\vc n}$ is a face of $\K_\rho$}\}.
\]
Then $\P=\E_{-1}^*$.  It is not clear that $\P$ is a packing (since the spheres may intersect), nor that it is dense.  We establish this for $n=4, 5$ and $6$ in \cite{Bar18}, and for $n=7$ and $8$ in \cite{Bar19}.  The packings are crystallographic and of lattice type.   For $n=4, 5$ and $6$, the group $\CC$ can be chosen to be a reflective group \cite{Bar19}.
\end{remark}

\begin{remark}  For a K3 surface $X$, let $\LL=\Pic(X)$ and $\P=\E_{-2}^*$, the set of irreducible $-2$ curves on $X$.  Then $\K_{\P}$ is the ample cone for $X$.  This was the motivation for the definitions given in \cite{Bar18}.

Note that the set $\P$ may not be a packing, since there can be pairs of $-2$ curves that do not intersect or intersect exactly once.  However, using a result of Morrison \cite{Mor84}, there exist plenty of K3 surfaces where this does not happen.  The packings defined in \cite{Bar18} and \cite{Bar19} can all be thought of as coming from K3 surfaces.
\end{remark}

\subsection{Enriques surfaces}\label{sAG}

For a nice introduction to Enriques surfaces, see \cite{Dol16b}.
Let $X$ be an Enriques surface and let $\LL$ be its Picard group modulo torsion.  Then $\LL$ is independent of the choice of $X$, and is sometimes called the Enriques lattice $E_{10}$.  It is an even unimodular lattice of signature $(1,9)$, and can be decomposed as the orthogonal product of the $E_8$ lattice (with negative definite inner product) and the plane $U$ equipped with the Lorentz product $\mymatrix{0&1 \\ 1&0}$: $\LL=E_8\oplus U$.  %We will describe E10 further in the next section.

A nice representation of $E_{10}$ is given by the Coxeter graph $T_{237}$ whose nodes have norm $-2$:
\begin{center}
\includegraphics{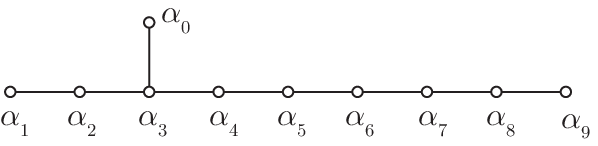}
\end{center}
The subscript in $T_{237}$ means the graph is a tree with three branches of length $2$, $3$, and $7$, as above.  The lattice is $\LL=\aa_0\Bbb Z\oplus \cdots \oplus \aa_9\Bbb Z$.  Since $\aa_i\cdot \aa_i=-2$, we know $\aa_i\cdot \aa_j=1$ if $\aa_i\aa_j$ is an edge of the graph, and $0$ otherwise.  Let $J=[\aa_i\cdot \aa_j]$, so for $\vc x,\vc y\in \LL$ written in this basis, $\vc x\cdot \vc y=\vc x^tJ\vc y$.  Since $J$ has integer entries and $-2$'s along the diagonal, the lattice is even, meaning $\vc x\cdot \vc x$ is even for all $\vc x\in \LL$.  One can verify $\det(J)=1$, so the lattice is unimodular.

The reflections
\[
R_{\aa_i}(\vc x)=\vc x -2\proj_{\aa_i}(\vc x)=\vc x-2\frac{\vc x \cdot \aa_i}{\aa_i\cdot \aa_i}\aa_i=\vc x+(\vc x\cdot \aa_i)\aa_i
\]
have integer entries, so are in $O_\LL^+$.
The inverse of $J$ (which appears in \cite[p.11]{Dol16b}) has non-negative entries along the diagonal, so the polytope bounded by the faces $H_{\aa_i}$ has finite volume.  Thus, the  Weyl group
\[
W_{237}=\langle R_{\aa_0},...,R_{\aa_{9}}\rangle
\]
has finite index in $O_{\LL}^+$.

\ignore{By norm $-2$ we mean $\aa_i\cdot \aa_i=-2$.  Each node represents a lattice point $\aa_i$ in $\LL$, and also a symmetry of the lattice given by reflection through the plane normal to $\aa_i$:
\[
R_{\aa_i}(\vc x)=\vc x -2\proj_{\aa_i}(\vc x)=\vc x-2\frac{\vc x \cdot \aa_i}{\aa_i\cdot \aa_i}\aa_i=\vc x+(\vc x\cdot \aa_i)\aa_i.
\]
If two nodes/vectors $\aa_i$ and $\aa_j$ are joined by an edge, then the angle between them is $2\pi/3$, so $\aa_i\cdot \aa_j=1$ (since $\aa_i\cdot\aa_i=\aa_j\cdot\aa_j=-2$).  Otherwise, the vectors are perpendicular, so $\aa_i\cdot \aa_j=0$.  Thus the Coxeter graph, together with the norms of the nodes, describes the intersection matrix $J=[\aa_i\cdot \aa_j]$.  One can verify that $\det(J)=1$, so $E_{10}=\aa_0\Bbb Z\oplus \cdots  \aa_{9}\Bbb Z$ is an even unimodular lattice.  (By even, we mean for any $\vc x\in E_{10}$, $\vc x\cdot \vc x$ is even.)   One can also check that $J^{-1}$ is the matrix in \cite[p. 11]{Dol16b}.   The group of symmetries for $E_{10}$ is the Weyl group
\[
W_{237}=\langle R_{\aa_0},...,R_{\aa_{9}}\rangle.
\]
It is clear that $W_{237}$ is contained in the group of symmetries of $E_{10}$, since $R_{\aa_i}$ preserves the lattice.  That it is the full group requires an investigation of its fundamental domain.

The lattice $\aa_0\Bbb Z \oplus \cdots \oplus \aa_7\Bbb Z$ is a copy of $E_8$ (with negative definite inner product), so $E_8$ is generated by the Coxeter graph $T_{235}$ with nodes of norm $-2$.
}

A generic Enriques surface has no {\it nodal} curves, meaning it has no smooth rational curves.  The moduli space of Enriques surfaces is ten dimensional.  If $X$ contains a nodal curve $\nu$, then we call it a {\it nodal Enriques surface}.  By the adjunction formula, $\nu\cdot \nu=-2$.  Here we have abused notation by letting $\nu$ represent both the curve on $X$ and its representation in $\LL$.  Since distinct irreducible curves on $X$ have non-negative intersection, an element of $\LL$ represents at most one nodal curve on $X$.  The moduli space of nodal Enriques surfaces is nine dimensional.    The following is a rewording of a portion of \cite[Theorem 1.2]{All18}:

\begin{theorem}[Coble, Looijenga, Cossec, Dolgachev, Allcock] \label{tAllcock}  Suppose $X$ is a generic nodal Enriques surface with nodal curve $\nu$.   Let $\LL$ be its Picard group modulo torsion.  Then there exist $\bb_0,..., \bb_9\in \LL$ so that $\bb_i\cdot\bb_i=-2$ and $\bb_1,...,\bb_9,\nu$ are the nodes of the Coxeter graph (\ref{C2}).  Let
\[
\CC=\langle R_{\bb_0},..., R_{\bb_9}\rangle \cong W_{246}.
\]
Then the image in $\LL$ of all nodal curves on $X$ is the $\CC$-orbit of $\nu$.
\end{theorem}

\section{The Apollonian packing in eight dimensions}\label{s2}

\begin{theorem}  The packing $\P_\bb=\CC(\nu)$ is Apollonian.
\end{theorem}
\begin{proof}  From Maxwell \cite{Max82}, we know that this is a Boyd-Maxwell packing, so all we must verify is that it is maximal and that it contains a cluster of ten mutually tangent spheres.  (Because $\CC$ acts transitively on $\P_\bb$, the weak Apollonian property implies the Apollonian property.)

We define $J_\bb=[\bb_i\cdot \bb_j]$, which has $-2$ along the diagonal, $1$ if $\bb_i\bb_j$ is an edge in Coxeter graph (\ref{C2}), and $0$ otherwise.  Taking the inverse,
\[
J_\bb^{-1}=\frac{1}{2}\mymatrix{
5&3&6&9&12&10&8&6&4&2 \\
3&0&2&4& 6& 5&4&3&2&1 \\
6&2&4&8&12&10&8&6&4&2 \\
9&4&8&12&18&15&12&9&6&3 \\
12&6&12&18&24&20&16&12&8&4 \\
10&5&10&15&20&15&12&9&6&3 \\
8&4&8&12&16&12&8&6&4&2 \\
6&3&6&9&12&9&6&3&2&1 \\
4&2&4&6&8&6&4&2&0&0 \\
2&1&2&3&4&3&2&1&0&-1
  },
\]
we find that the weights all have non-negative norm except $w_9$, which gives us $\nu=2w_9=[2,1,2,3,4,3,2,1,0,-1]$.  Since $w_9\cdot \bb_i=\dd_{i9}$, we get $\nu\cdot\bb_9=2$ and $\nu\cdot \bb_i=0$ for $i\neq 9$, as expected.  This is our first sphere, which we label $s_0=\nu=2w_9$.  We get $s_1$ by reflecting $s_0$ across $H_{\bb_9}$:  $s_1=R_{\bb_9}(s_0)$.  We define $s_i$ through $i=9$ recursively by reflecting in subsequent planes: $s_{i+1}=R_{\bb_{9-i}}(s_i)$.  It is straight forward to verify that $s_i\cdot s_j=2$ if $i\neq j$, and of course, $s_i\cdot s_i=\nu\cdot\nu=-2$, so the set $\{s_0,...,s_9\}$ is a cluster of ten mutually tangent spheres.  

To show it is maximal, we can appeal to arguments like those in \cite{Bar18} and \cite{Bar19}.  Let us instead appeal to Maxwell's Theorem 3.3 \cite{Max82}.  By this result, it is enough to show that all weights $w_i$ are in the convex closure of $\CC(w_9)$.  Following Maxwell's example, we note
\[
w_{i-1}=w_i+R_{\bb_i}\cdots R_{\bb_9}w_9=(s_0+...+s_{10-i})/2,
\]
for $5\leq i \leq 9$, so these are all in the convex hull of $\CC(w_9)$.  This leaves us with four more to check, of which $w_1$ is different, as it is a cusp and no spheres go through it.  One can verify that $R_{w_9-\bb_9}\in\CC$, and that 
\begin{equation}\label{e3}
\lim_{k\to \infty}\frac{1}{k^2}(R_{w_9-\bb_9}\circ R_{\bb_9})^k(w_9)=w_1
\end{equation}
so $w_1$ is in the convex closure of $\CC(w_9)$.  This is a messy calculation; we will rationalize why it works in Remark \ref{r5}.  Finally, we note that 
\begin{align*}
w_2&=(s_{10}+R_{\bb_2}(s_8))/2\\
w_0&=w_1+s_9/2 \\
w_3&=R_{\bb_2}(w_2)+R_{\bb_1}(w_1).  
\end{align*}

Thus, $\P_\bb$ is an Apollonian sphere packing.  
\end{proof}

\subsection{The strip version and the $E_7$ lattice}  To better understand this packing, we describe it in a couple of ways.  Let us begin with a strip version, which is an analog of \fref{figApc}.   We can think of \fref{figApc} as an infinite set of circles of constant diameter that are each centered at lattice points of a one-dimensional lattice, then sandwiched between two lines, and filled in using the symmetries of the lattice together with inversion in $\aa_3$ and reflection in $\aa_4$.  We can describe $\P_\bb$ in a similar way:

\begin{theorem}\label{t2.1}  Consider the $E_7$ lattice imbedded in a 7-dimensional subspace $V$ of $\Bbb R^8$.  Centered at each lattice point, place a hypersphere of radius $1/\sqrt{2}$, and bound these hyperspheres by two hyperplanes parallel to $V$ and a distance $1/\sqrt{2}$ away from $V$, so that all the spheres are tangent to the two hyperplanes.  Let the sphere centered at $\vc 0$ be tangent to the two hyperplanes at $A$ and $B$, and let $\ss$ be inversion in the hypersphere centered at $A$ that goes through $B$.  Consider the image of these spheres in the group generated by the symmetries of the $E_7$ lattice, the inversion $\ss$, and refection across the hyperplane $V$.  The resulting configuration of hyperspheres is a perspective of $\P_\bb$.
\end{theorem}

\begin{remark}  The $E_7$ lattice has the Coxeter diagram $T_{234}$:
\[
\includegraphics{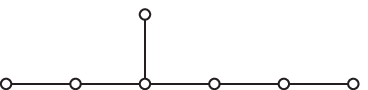}
\]
It is common to assign each node the norm $2$ so that the resulting bilinear form has integer entries, is even, and has determinant $2$.  Thus, the generating vectors have length $\sqrt 2$, which is why the spheres above have radius $\sqrt 2/2$.
\end{remark}

\begin{proof}  Referring to $J_\bb^{-1}$, we note that $w_8$ has norm zero, so is a point on $\partial \H$.  The plane $H_{\bb_9}$ goes through it, as do $H_{\bb_i}$ for $i=0,...,7$, while the latter are all perpendicular to $H_{\bb_9}$.  We let $w_8$ be the point at infinity, so $w_8$ is the analog of $E$ in \fref{figApc}.  We think of $\bb_9$ as the analog of $\aa_4$, so $V=H_{\bb_9}$.  The two spheres $s_0$ and $s_1$ also go through $w_8$ (so are the two hyperplanes) and indeed $4w_8=s_0+s_1$, so it is the point of tangency of these two spheres.

The planes $H_{\bb_i}$ for $i=0,...,6$ are all perpendicular to the sphere $s_2$, so let us think of the center $O$ of $s_2$ as the origin for $\Bbb R^8\cong \partial \H_{w_8}$.  We consider the group $\GG_7=\langle R_{\bb_0},...,R_{\bb_7}\rangle$, and set $\LL_7$ to be the image of $O$ under the action of $\GG_7$.

The portion of the Coxeter graph (\ref{C2}) with nodes $\bb_0,...,\bb_6$ is the Coxeter graph $T_{234}$ of the $E_7$ lattice, so it should be no surprise that $\LL_7$ is the $E_7$ lattice.  The group $\CC_7\cong W_{244}$ acts transitively on the $E_7$ lattice.  This is no doubt a well known result, but it is not hard to verify:  We note that the plane given by $\bb_i'=2w_8-\bb_i$  (for $0\leq i\leq 6$) is parallel to $H_{\bb_i}$ (since $\bb_i'\cdot \bb_i'=-2$ and $\bb_i'\cdot \bb_i=2$), and is tangent to $s_2$ (since $s_2\cdot \bb_i'=2$).  Thus the composition $R_{\bb'_i}\odot R_{\bb_i}$ is translation in the direction of $\bb_i$.  These translations, acting on $O$, generate the $E_7$ lattice.  To see that it is a subset of $\LL_7$, we show that $R_{\bb_i'}\in \CC_7$, which we do using a method of descent.

Our method of descent is as follows:  Given a vector $\vc n$, we descend to $R_{\bb_j}(\vc n)$ if $\vc n\cdot \bb_j<0$ (for $j\leq 7$).  Each $\bb_k'$ descends to a $\bb_i$ with $0\leq i\leq 7$, so $R_{\bb_k'}\in \CC_7$.

We note that $R_{\bb_i}$ fixes the $E_7$ lattice for $i\leq 7$, so $\LL_7$ is the $E_7$ lattice.

Finally, the generators of $\CC$ are $R_{\bb_8}$, $R_{\bb_9}$, and the generators of $\CC_7$.  The map $R_{\bb_9}$ is reflection across the hyperplane $V$.  The reflection $R_{\bb_8}$ is $\ss$, as it sends the hyperplane $s_1$ to the sphere $s_2$, so is inversion in the sphere through the point of tangency of $s_2$ and $s_1$, and centered at the point of tangency of $s_2$ and $s_0$.  Thus, the described packing is $\P_{\bb,w_8}$.
\end{proof}

\subsection{The $E_8$ structure}
\begin{theorem}
Consider an arrangement of hyperspheres all of radius $1/\sqrt 2$ and centered at the lattice points of the $E_8$ lattice in $\Bbb R^8$.  (This is the densest uniradial sphere packing in $\Bbb R^8$.)  In this configuration, there is a tight cluster of nine spheres such that there exists a sphere of radius $1/\sqrt 2$ (not in the arrangement)  that is perpendicular to all nine spheres.  Let $\ss$ be inversion in this last sphere.  Consider the image of this arrangement of spheres in the group generated by $\ss$ and the symmetries of the $E_8$ lattice.  The resulting configuration is a perspective of $\P_\bb$.
\end{theorem}

\begin{proof}  Referring to $J_\bb^{-1}$, we see that $w_1$ is also on $\partial \H$, as $w_1\cdot w_1=0$.  Let us pick $w_1$ as our point at infinity.  Then the planes $H_{\bb_i}$ for $i\neq 1$ all go through $w_1$ (by the definition of $w_1$), so the corresponding reflections $R_{\bb_i}$ are Euclidean symmetries of $\Bbb R^8\cong \partial H_{w_1}$.
Note that $\bb_i\cdot \nu=\bb_i\cdot s_0=0$ for $i\neq 1,9$.  We choose the center $O$ of $s_0$ for the origin of $\Bbb R^8$.  Note that the subgraph of (\ref{C2}) with nodes $\{\bb_0, \bb_2,...,\bb_8\}$ is the $T_{235}$ graph, which is the Coxeter graph for the $E_8$ lattice.

Let $\CC_8=\langle R_{\bb_0},R_{\bb_2},...,R_{\bb_9}\rangle$ and set $\LL_8=\CC_8(O)$.  Then $\CC_8\cong W_{236}$, so $\LL_8$ is the $E_8$ lattice.  This is a better known result, but is also easy to verify using the same ideas as in the proof of Theorem \ref{t2.1} (using $\bb_i'=w_1-\bb_i$ for $i\neq 1$).  The image $\CC_8(s_0)$ is a uniradial sphere packing and is the arrangement referred to in the statement of the Theorem.

The group $\CC$ is generated by $\CC_8$ and the reflection $R_{\bb_1}$, which is an inversion on $\Bbb R^8$.  The spheres $\{s_0,...,s_8\}$ are all in $\CC_8(s_0)$, and $\bb_1$ is perpendicular to all but $s_8$.  Let $s_{10}=R_{\bb_0}(s_6)$.  Then $\bb_1$ is perpendicular to the nine spheres $\{s_0,s_1,...,s_7,s_{10}\}$.  This is the tightly clustered set of spheres referred to in the statement of the theorem.  Note that $\bb_1\cdot w_1=2=\nu \cdot w_1$, so the radii of $H_{\bb_1,w_1}$ and $H_{\nu,w_1}$ are the same.   Thus, $R_{\bb_1}$ is inversion in a sphere of radius $1/\sqrt 2$ in $\partial \H_{w_1}$.

Thus, the described configuration is congruent to the perspective $\P_{\bb,w_1}$.
\end{proof}

\begin{remark}\label{r5}  It is clear that $w_1$ is a limit point of the spheres on the $E_8$ lattice.  The limit in Equation (\ref{e3}) is the limit of spheres translated along the $\bb_9$ direction.
\end{remark}

\begin{remark}  Note that the above shows that there is only one type of hole in the $E_8$ uniradial sphere packing.  Compare this to the usual cannon-ball sphere packing in $\Bbb R^3$, where there are two types of holes:  One in a tetrahedral configuration of spheres, and another in an octahedral configuration of spheres.
\end{remark}

\subsection{Comparison with the packing in \cite{Bar18}}
Let
\begin{align*}
\LL_{\bb}&=\bb_0\Bbb Z\oplus \cdots \oplus \bb_9\Bbb Z \\
\LL_{10}&=s_0\Bbb Z\oplus \cdots \oplus s_9\Bbb Z
\end{align*}
and let $J_{10}=[s_i\cdot s_j]$.  Then $\det(J_{10})=-2^{22}$ while $\det(J_\bb)=-4$.  It would appear that the packing $\P_{10}=\E_{-2}^*$ generated by $\LL_{10}$ (see Remark \ref{r2}) is very different than $\P_\bb$.  However, we can modify the underlying lattice for $\P_\bb$ in the following way.  Consider the sublattice
\[
\LL_{\bb}'=2\bb_0\Bbb Z\oplus \cdots \oplus 2\bb_8\Bbb Z\oplus \nu\Bbb Z
\]
and its incidence matrix $J_{\bb}'$.
Then $\LL_{10}\subset \LL_{\bb}'\subset \LL_{\bb}$ and $\det(J_{\bb}')=-2^{20}$.  Thus $\LL_{10}$ is a sublattice of $\LL_{\bb}'$ of index two. Though we have changed the underlying lattice for $\P_\bb$, we have not changed its geometry.  Furthermore, in this basis, $\P_\bb=\E_{-2}^*$.  Thus, $\K_{\P_{10}}\supset \K_{\P_\bb}$, so the residual set for the packing $\P_\bb$ is a subset of the residual set for $\P_{10}$.  Thus, $\P_\bb$ is a more {\it efficient} packing than $\P_{10}$.

\begin{remark}  There exists a K3 surface $X$ with $\Pic(X)=\LL_{\bb}'$, so we can think of $\P_\bb$ as the ample cone for some K3 surface.
\end{remark}

\section{The Apollonian property and Maxwell's paper}\label{s3}

In this section, we identify the relevant property of the Coxeter graph (\ref{C2}) that implies the existence of a maximal cluster of mutually tangent spheres, and use this to identify the Apollonian sphere packings in Maxwell's paper \cite{Max82}.

\begin{theorem} \label{t3.1} Let $\P=\CC(S)$ be a crystallographic sphere packing.  Suppose there exists an element $\nu$ of $\P$ and a reflective subgroup of $\CC$ that generates a Coxeter graph of the form $T_{k}$ with $\nu$ attached to one of the ends by a bold edge.  Then $\P$ includes a cluster of $k+1$ mutually tangent spheres.
\end{theorem}

\begin{proof}  Let us label the graph $T_k$ and $\nu$ as follows:
\[
\includegraphics{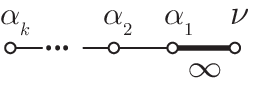}
\]
and set $\aa_i\cdot \aa_i=\nu\cdot \nu=-2$.  Let $s_1=\nu$ and $s_{i+1}=R_{\aa_i}(s_i)$ for $i=1, ... , k$.  Then $s_i\cdot s_i=-2$ for all $i$.  To show the set $\{s_1,...,s_{k+1}\}$ is a cluster of $k+1$ mutually tangent spheres, we need to show $s_i\cdot s_j=2$ for $i\neq j$.  We prove this using induction on the following statements:
(a) $s_i=s_{i-1}+2\aa_{i-1}$;
(b) $\aa_i\cdot s_i=2$;
(c) $\aa_j\cdot s_i=0$ for $j>i$; and (d)  $s_i\cdot s_j=2$ for $i<j$.
We leave the details to the reader.  
\end{proof}

While this result may not be surprising, and may even be obvious, it nevertheless seems to have escaped any serious notice.  With this result in mind, we look at Maxwell's Table II \cite{Max82} and Chen and Labb\'e's Appendix \cite{CL15} in search of candidates for Apollonian packings.  We find twelve candidates:  The one in the introduction and the eleven listed in \fref{fig2b}.
\myfig{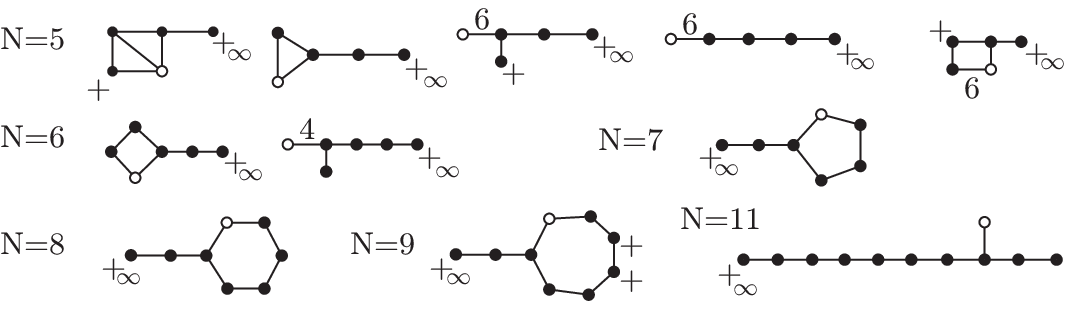}{The eleven other graphs in \cite{CL15} that satisfy the conditions of Theorem \ref{t3.1} for $k=N-1$.}
The hollow node indicates that removing it yields the desired subgraph.  The $+$ is Maxwell's notation to indicate that the weight at that point is a plane.  We have modified his notation a bit, using $+_\infty$ to indicate that the weight is a plane that is parallel to its associated node, while $+$ indicates that the plane is ultraparallel to its associated node.

Maxwell notes that the first four graphs for $N=5$ yield the same sphere packing \cite[Table I]{Max82}, which is the Soddy sphere packing \cite{Sod37}.  The two graphs in $N=6$ also give the same packing.  The packings for $N=6$, $7$, and $8$ are the subject of \cite{Bar18}, and their Coxeter graphs are given in \cite{Bar19}.  The graphs for $N=9$ and $N=11$ are the subjects of the next two sections.

The last graph for $N=5$ is a pleasant surprise, as it shows that there are different ways of filling in the voids of an initial configuration of five mutually tangent spheres in $\Bbb R^3$, yet still get a sphere packing where every sphere is a member of a cluster of five mutually tangent spheres.  The packing is Apollonian, but not lattice like.  Unlike the Soddy packing, there is no perspective where all the spheres have integer curvature.  

Let us label the Coxeter graph as follows:
\begin{equation}\label{CoxMax}
\includegraphics{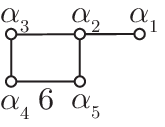}
\end{equation}
so
\[
J_\aa=\mymatrix{-2&1&0&0&0 \\ 1&-2&1&0&1 \\ 0&1&-2&1&0 \\ 0&0&1&-2&\sqrt 3 \\ 0&1&0&\sqrt 3&-2}.
\]
We let $s_1=w_1$, $s_2=R_{\aa_1}(s_1)$, ..., $s_5=R_{\aa_4}(s_4)$.  We note $w_2=s_1+s_2$ and let it be the point at infinity, giving us a strip version of the packing.  The cross section on the plane $\aa_1$ is shown in \fref{fig7}, and the cross section on the plane $\aa_4$ is shown in \fref{fig8}

\begin{figure}
\includegraphics[height=7.4cm]{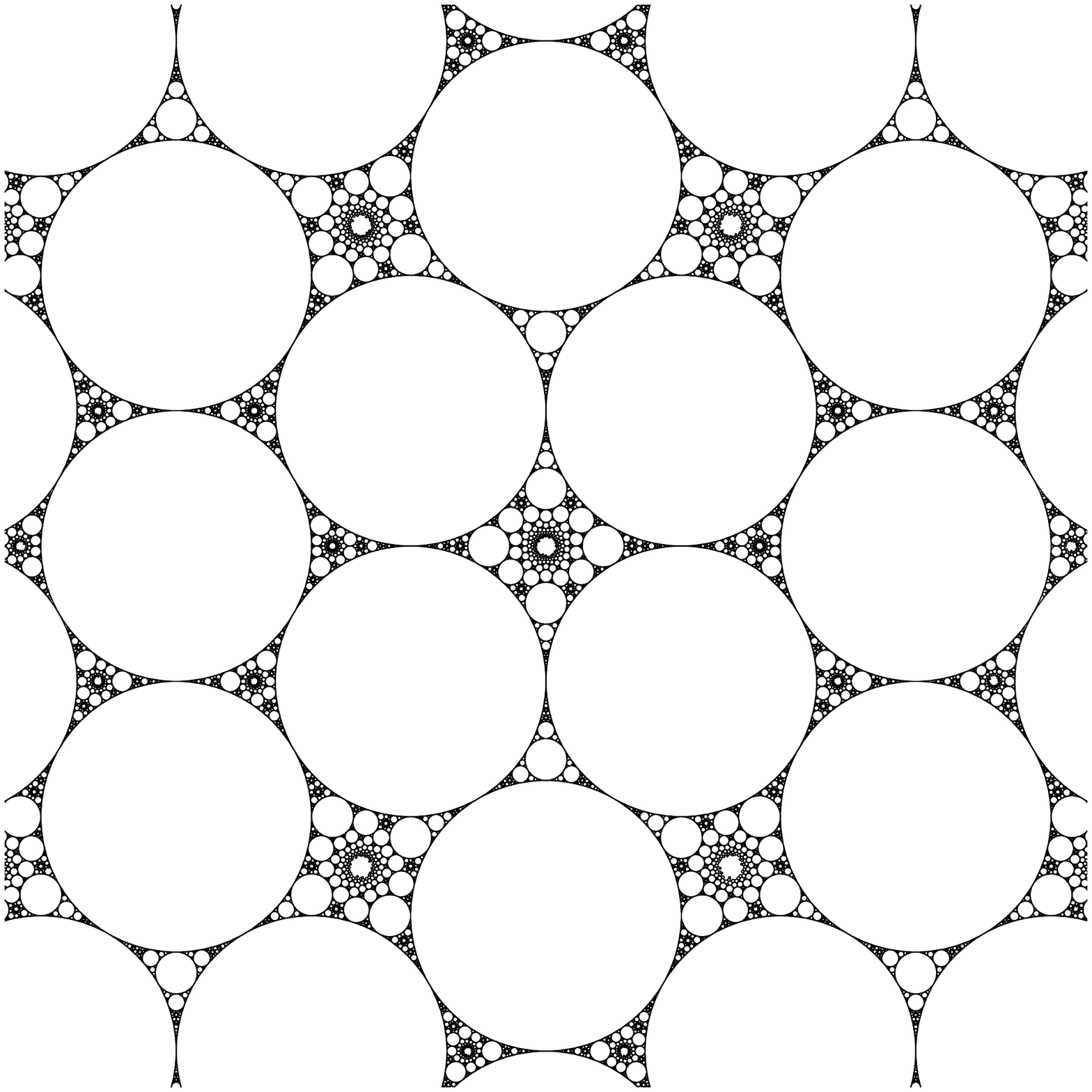}

\vspace{-7.43cm}\hspace{-.1cm}
\includegraphics[height=7.45cm]{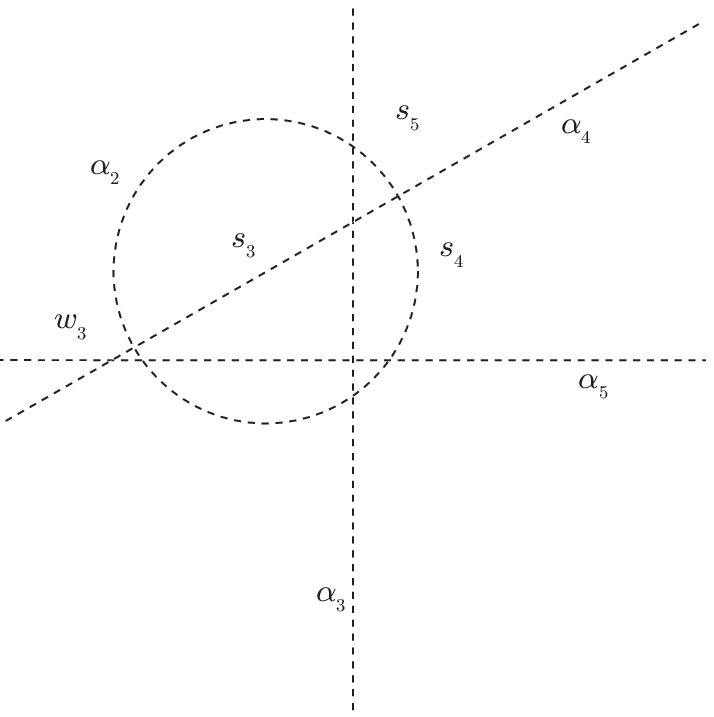}
\caption{\label{fig7}  The horizontal cross section of the strip version of the non-Soddy Apollonian sphere packing with Coxeter graph (\ref{CoxMax}).  The dotted lines represent the symmetries.  Note that the dotted circle represents inversion in a sphere that intersects this plane at an angle of $\pi/3$.  This picture \fref{fig8} were generated using McMullen's Kleinian groups program \cite{McM}.}
\end{figure}

\begin{figure}
\includegraphics[width=\textwidth]{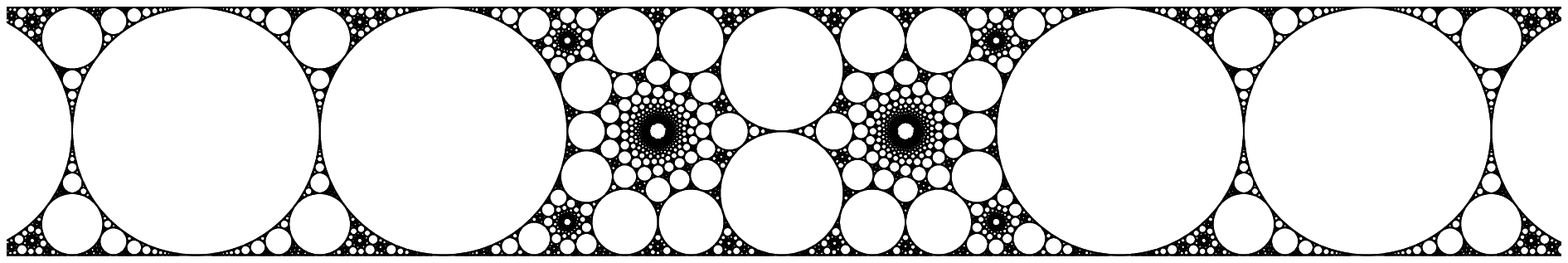}

\vspace{-2.09cm}
\includegraphics[width=\textwidth]{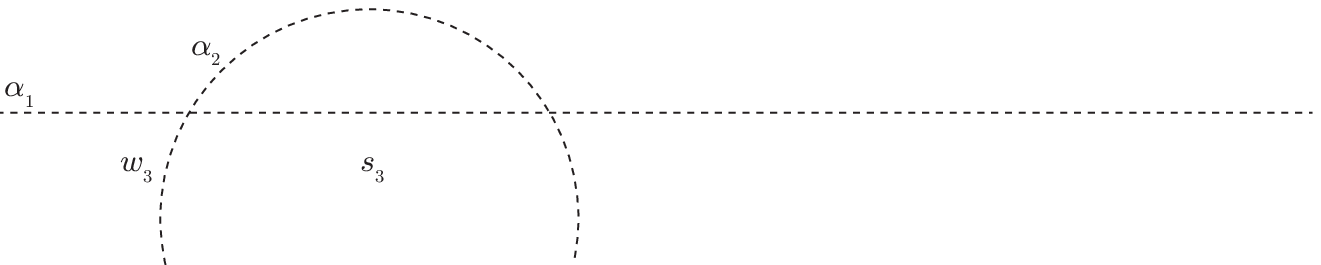}
\caption{\label{fig8}  The cross section on the plane $H_{\aa_4}$ of the strip version of the non-Soddy Apollonian sphere packing.  As a circle packing, this cross section is weakly Apollonian.  }
\end{figure}

\begin{remark}  If we invert the strip packing in the sphere $w_3$ and scale to get a sphere of curvature $-1$, then the planes $s_1$ and $s_2$ become spheres with curvature $2$ tangent at the center of the sphere $w_3$.  The six spheres surrounding $w_3$ become a {\it hexlet} of spheres with curvature $3$, just like those in the sphere packing described by Soddy \cite{Sod37}.  How the space between these spheres is filled in, though, is different from how Soddy does it.  The spheres $s_3$ and $s_4$, in this perspective, have curvature $7+4\sqrt 3$.
\end{remark}

\section{A cross section in $\Bbb R^7$}  \label{s4}

Given a sphere packing in dimension $n$, a codimension one cross section is a sphere packing in dimension $n-1$.  If the sphere packing contains a cluster of $n+2$ mutually tangent spheres (i.e.~has the Apollonian property), then by choosing a cross section perpendicular to $n+1$ of these spheres, we get a sphere packing in one lower dimension that has (at least) the weak Apollonian property.

\begin{theorem}  A cross section perpendicular to nine spheres in a cluster of ten mutually tangent spheres in $\P_\bb$ yields the sphere packing in $\Bbb R^7$ with Coxeter graph labeled $N=9$ in \fref{fig2b}.
This packing is of general lattice type.
\end{theorem}

\begin{proof}
Let $H$ be the plane perpendicular to $s_i$ for $i=0$, ..., $8$.  Then $H$ has normal vector
\[
h=[-1,3,2,1,0,0,0,0,0,0].
\]
Note that $h\cdot \bb_i=0$ for $i=2$,...,$9$.  Knowing the expected Coxeter graph, we solve for $\bb_1'$ so that we get:
\[
\includegraphics{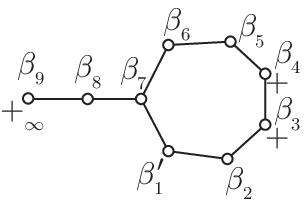}
\]
Thus we want $\bb_1'\cdot \bb_i=0$ for $i\neq 0$, $2$, or $7$; $\bb_1'\cdot h=0$; $\bb_1'\cdot \bb_2=\bb_1'\cdot \bb_7=1$; and $\bb_1'\cdot \bb_1'=-2$.  We get $\bb_1'=[2,1,1,2,3,2,1,0,0,0]$.  Using descent, we show $R_{\bb_1'}\in \CC$.  The set $\{\bb_1',\bb_2,...,\bb_9\}$ forms a basis of the subspace $H$.  Let $J_7$ be the incidence matrix for this basis of $H$, and let $w_i'$ be the weights.  Since $s_0\cdot h=0$, we get $s_0=w_9'$.  Let $\CC_7=\langle R_{\bb_1'},R_{\bb_2},...,R_{\bb_9}\rangle$.  Then the spheres in $\CC_7(s_0)$ all intersect $H$ perpendicularly.
Most spheres in $\P_\bb$ miss $H$ and some are tangent to $S$,  but there are some that intersect $H$ at an angle that is not right.
In particular, let $s_{10}=R_{\bb_0}(s_6)$ and $s_{11}=R_{\bb_9}(s_{10})$.  Then $s_{10}\cdot h=4$ and $s_{11}\cdot h=-4$, while $h\cdot h=-14$, so these two spheres intersect $H$ but not perpendicularly nor tangentially.  The difference in signs (the $\pm 4$) indicates that the centers of the spheres are on opposite sides of $H$.  The intersection of $s_{10}$ with $H$ is found by projecting the vector $s_{10}$ onto $H$ to get
\[
n=s_{10}-\frac{s_{10}\cdot h}{h\cdot h}h.
\]
We note that $n\cdot h=0$, $n\cdot \bb_1'=0$, and $n\cdot \bb_i=0$ for $i\neq 0$, $1$, or $4$, so $n$ is a scalar multiple of $w_3'$.  Similarly, the projection of $s_{11}$ onto $H$ is a scalar multiple of $w_4'$.  Thus, the intersection of $\P_\bb$ with $H$ is the packing
\[
\P_7=\CC_7(\{s_0,w_3',w_4'\}).
\]
The spheres in $\CC_7(s_0)$ in the perspective with $w_8$ the point at infinity  (or any other lattice point for the point at infinity) all have curvature an integer multiple of $\sqrt 2$, while those in $\CC_7(w_3')$ and $\CC_7(w_4')$ all have curvature an integer multiple of $\sqrt{42}/3$.  Thus, the packing is not of lattice type, though it is of general lattice type, since all spheres have integer coordinates in the basis $\{\bb_1',\bb_2,...,\bb_9\}$.
\end{proof}

\begin{remark}  Every sphere in $\CC_7(s_0)$ is a member of a cluster of $9$ mutually tangent spheres.  The same cannot be said of the spheres in the orbits of $w_3'$ and $w_4'$.
\end{remark}

\begin{remark}  A similar cross section of the (known) Apollonian packings in dimensions $n\leq 6$ intersect all spheres perpendicularly, so give the Apollonian packing in one dimension lower.
\end{remark}

\begin{remark}  The spheres $s_{10}$ and $s_{11}$ are tangent at the point $P_1=s_{10}+s_{11}$, which lies in the subspace $H$ and is on $\partial \H$.  Let $P_2=R_{\bb_3}(P_1)$.

For two points $A$ and $B$ on $\partial \H$, define the {\it Bertini involution} to be
\[
\phi_{A,B}(\vc x)=2\frac{(A\cdot \vc x)B+(B\cdot \vc x)A}{A\cdot B}-\vc x.
\]
This is the map that is $-1$ on $\partial \H_A$ through the point $B$.

The Bertini involutions $\phi_{P_1,w_8}$ and $\phi_{P_1,P_2}$ both preserve the lattice $\bb_1'\Bbb Z\oplus \bb_2\Bbb Z\oplus \cdots \oplus \bb_9\Bbb Z$.  
Note that $\phi_{P_1,P_2}(w_4')=-w_4'$, so $\phi_{P_1,P_2}$ sends everything on one side of the plane $H_{w_4'}$ to the other side.  The map $\phi_{P_1,w_8}$ sends $w_3'$ to $w_4'$.  Let
\[
\CC_7'=\langle \CC_7,\phi_{P_1,P_2},\phi_{P_1,w_8} \rangle.
\]
Since $\CC_7(s_0)$ does not intersect $H_{w_3'}$ or $H_{w_4'}$, $\CC_7'(s_0)$ is a sphere packing.  Because $\P_7$ is maximal, $\CC_7'(s_0)$ is also maximal.  It is lattice like and Apollonian.  Since $\CC_7\leq  \CC_7'$, the limit set of $\CC_7$ is a subset of the limit set of $\CC_7'$.  Since these are the residual sets of the respective packings, the packing $\P_7$, which is only weakly Apollonian, is more efficient than the Apollonian packing $\CC_7'(s_0)$.  

The packing $\CC_7'(s_0)$ is the packing described in \cite{Bar19}.  The cone $\K_{\CC_7'(s_0)}$ is the ample cone for a class of K3 surfaces.  The packing $\P_7$, though, cannot be the ample cone of any K3 surface.  
\end{remark}

\section{A sphere packing in $\Bbb R^9$}\label{s5}

The sphere packing generated by the Coxeter graph with $N=11$ in \fref{fig2b} can be described as follows:  Let $H$ be an 8-dimensional subspace of $\Bbb R^9$ (with normal vector $h$) and let us place spheres of radius $1/\sqrt 2$ at each vertex of a copy of the $E_8$ lattice imbedded in $H$.  Let us place two hyperplanes parallel to $H$ a distance of $1/\sqrt 2$ on either side, so that they are tangent to all the spheres.  Let us distinguish the sphere $s_0$ centered at the origin of the $E_8$ lattice and let its points of tangencies with the hyperplanes be $A$ and $B$.  Let $\ss$ be inversion in the sphere centered at $A$ and through $B$.  The packing is the image of these spheres under the group generated by $\ss$, the symmetries of the $E_8$ lattice, and reflection $R_h$ in $H$.

To see this, recall that the spheres centered on the vertices of the $E_8$ lattice are generated by the image of $s_0$ under the action of the Weyl group $W_{236}$, which is $\langle R_{\bb_0},R_{\bb_2},...,R_{\bb_9}\rangle$ for the $T_{236}$ subgraph of Graph (\ref{C2}).  The planes $H_{\bb_i}$ are perpendicular to $s_0$ for $i\neq 9$, and hence also perpendicular to $\ss$.  (We ignore $i=1$ in this discussion.)  Note that $H_{\bb_9}$ is tangent to $s_0$.  Let $s_1=R_{\bb_9}(s_0)$, giving us the cross section shown in \fref{fig4b}.
\myfig{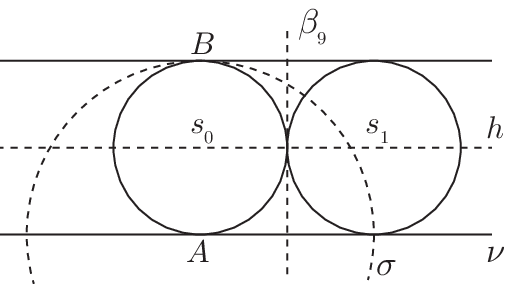}{A cross section perpendicular to $H$, $s_0$, and $s_1$.}
We note that $\ss$ and $\bb_9$ are at an angle of $2\pi/3$.  We note that $h$ is perpendicular to $\bb_i$ for all $i$, and $h$ and $\ss$ intersect at an angle of $2\pi/3$.  Finally, let $\nu$ represent the plane tangent to $s_0$ at $A$.  Then $\nu$ is parallel to $h$, perpendicular to $\ss$, and perpendicular to $\bb_i$ for all $i$.  This gives us the Coxeter graph
\[
\includegraphics{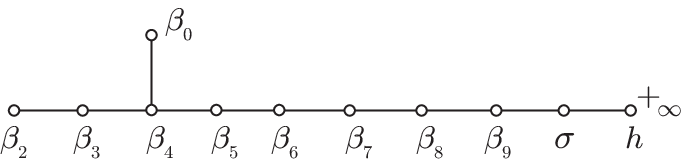}
\]
as desired, where $\nu$ is the weight at the node  $h$.  Maxwell verifies that this packing is maximal (see the discussion after Theorem 3.3 in \cite{Max82}).

\begin{remark}  Maxwell also presents the packing with Coxeter graph
\[
\includegraphics{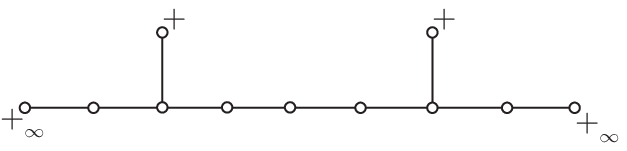}
\]
which generates the same packing.  Maxwell identifies an invariant of packings of lattice type and notes that these two packings have the same invariant, but presumably could not show equivalence.
\end{remark}

\end{document}